\documentclass[11pt]{amsart}
\usepackage{amssymb}
\usepackage{graphicx} 
\usepackage{enumerate}
\usepackage{mathtools}
\usepackage{color,soul}
\usepackage{cite}
\usepackage{tikz}
\usetikzlibrary{matrix}

\usepackage{amsmath, amsthm}
 \usepackage{relsize}

\usepackage{color}

\definecolor{lgray}{gray}{0.82}
\definecolor{rredd}{rgb}{120,0,0}

\usepackage[margin=1.2in]{geometry}

\newtheorem{thm}{Theorem}[section]
\newtheorem{prop}[thm]{Proposition}
\newtheorem{lem}[thm]{Lemma}
\newtheorem{cor}[thm]{Corollary}
\newtheorem{defn}[thm]{Definition}

\newtheorem{ex}[thm]{Example}

\newcommand{\theoremname}{Theorem:}

\newcommand{\R}{\mathbb{R}}
\newcommand{\Z}{\mathbb{Z}}

\theoremstyle{definition}

\theoremstyle{remark}

\numberwithin{equation}{section}

\title{ Classification of the Discrete Real Specializations of the Burau representation of $B_3$}
\author{Nancy Scherich}

\begin{document}

\begin{abstract}
This classification is found by analyzing the action of a normal subgroup of $B_3$ as hyperbolic isometries. This paper gives an example of an unfaithful specialization of the Burau representation on $B_4$ that is faithful when restricted to $B_3$, as well as examples of unfaithful specializations of $B_3$.
\end{abstract}
\maketitle
\section{Introduction}

Representations of the braid groups have attracted attention because of their wide variety of applications from discrete geometry to  quantum computing.  One well studied representation is the Burau representation. Thanks to the work of Moody \cite{MOODY}, Long and Paton \cite{LONGPATON}, and Bigelow \cite{BIG}, the Burau representation is famously known for its longstanding open question of faithfulness for $n=4$. This paper takes the point of view that one should also ask other structural questions about the image of a braid group representation, in particular whether the image is discrete for specializations of the parameter.  Venkataramana in \cite{VEN} also followed this pursuit for discrete specializations of the Burau representation but with a different approach toward arithmeticity.

The Burau representation is one summand of the Jones representations, which are used in modeling quantum computations. So, much work has been done to understand specializations to roots of unity, as explored by Funar and Kohno in \cite{FK2014}, Venkataramana in \cite{VEN}, Freedman, Larson and Wang in \cite{FLW}, and many others. However, there seems to be a lack of exploration of the real specializations of the Burau representation, which is the topic of this paper.  The main theorem we will prove is a complete classification of the discrete real specializations of the Burau representation for $n=3$ as well as some faithfulness results.




\begin{thm}\label{combind}
The real discrete specializations of the Burau representation  of $B_3$ are exactly when $t$ satisfies one of the following:
\begin{enumerate}
\item $t<0$ and $t\neq -1$
\item $0<t\leq\frac{3-\sqrt{5}}{2}$ or $t\geq\frac{3+\sqrt{5}}{2}$
\item $\frac{3-\sqrt{5}}{2} <t<\frac{3+\sqrt{5}}{2}$ and the image forms a triangle group.
\end{enumerate}
Additionally, the specialization is faithful in $(1)$ and $(2)$.
\end{thm}

It is well known that the Burau representation of $B_3$ is faithful for transcendental $t$ and is unfaithful for $t=-1$, see\cite{BMP}. In particular, Theorem \ref{combind} shows that $t=-1$ is the only real negative unfaithful specialization and all other real unfaithful specializations of the Burau representation of $B_3$ come from the interval $(\frac{3-\sqrt{5}}{2}, \frac{3+\sqrt{5}}{2})$. A short solvability argument in Section 4 leads to a construction of  unfaithful specializations described in the following corollary.\\

\textbf{Corollary 4.2:}\textit{ Let $\alpha $ be a positive real root of the 2-1 entry of a matrix in the image of the Burau representation of $B_3$ not in $\langle\sigma_1\rangle$. Then $\alpha\in(\frac{3-\sqrt{5}}{2}, \frac{3+\sqrt{5}}{2})$ and specializing $t=\alpha$ is an unfaithful specialization of the Burau representation.}\\

One motivation for studying discrete representations of the braid group is the peculiar relationship between discreteness and faithfulness.  A restated version of Wielenberg's theorem  shows how a sequences of discrete representations of the braid group can give rise to a faithful representation \cite{WIE}. \\
 
\noindent \textbf{Wielenberg's Theorem:}\textit{
 Let $\{\rho_i\}$ be a sequence of matrix representations of $B_n$ that are discrete and co-finitely faithful. If $\{\rho_i\}$ converges algebraically to a representation $\rho$, then $\rho$ is both discrete and faithful.}\\

While Wielenberg's Theorem cannot be directly applied with Theorem \ref{combind} to give progress on the faithfulness of Burau for $n=4$,  it gives motivation for studying discreteness of a representation as a possible route to faithfulness.  For contrast, the BMW representations are known to be faithful \cite{ZIN,BIG2001}, but not much is known about the discreteness of the image.  There is still much to be explored about the relationship between faithfulness and discreteness of representations of the braid groups.

 Additionally, there is an interesting interplay of the faithfulness of the Burau representation for $B_3$ and $B_4$. The characterization from Theorem \ref{combind} can be used to create unfaithful specializations for $n=4$ that are faithful when restricted to $n=3$. An example of such a specialization is given in Section 4, in addition to examples of constructing unfaithful specializations for $n=3$. While Theorem \ref{combind} is worthy in its own right, these corollaries and examples are particularly taunting in light of the  faithfulness question for $n=4$.

\section{Background Information and Notation}

\subsection{}\textbf{The Burau Representation}

Let $\sigma_i$'s be the standard generators of $B_n$, the braid group on $n$ strands.
\begin{defn}The Burau representation of $B_3$ is the homomorphism $\rho_3:B_3\rightarrow GL_2(\mathbb{Z}[t,t^{-1}])$ given by
 \[ \rho_3(\sigma_1)=\left( \begin{array}{cc}
-t & 1  \\
0 & 1  \end{array} \right)\]

\[\rho_3(\sigma_2)=\left( \begin{array}{cc}
1 & 0  \\
t & -t  \end{array} \right).\] 

\end{defn}

The Burau representation is defined similarly for the braid group on arbitrary $n$ strands, denoted $\rho_n:B_n\rightarrow GL_{n-1}(\mathbb{Z}[t,t^{-1}]).$
Squier showed in \cite{Sq} that there exists a nonsingular $(n-1)\times (n-1)$ matrix $J$ over $\mathbb{Z}[t,t^{-1}]$ so that for every $w$ in $B_n$, $$\rho_n(w)^*J\rho_n(w)=J.$$  
 \textbf{Notation:}   For $M\in GL_{n-1}(\mathbb{Z}[t,t^{-1}])$, the entries of $M$ are integral polynomials in $t$ and $t^{-1}$, and we denote $M=M(t)$ and $M(t^{-1})$ to be the matrix that replaces $t$ by $t^{-1}$ in the entries of $M(t)$. The involution $*$ is given by $M(t)^*=M(t^{-1})^T$.

\begin{defn} A \textbf{specialization of the Burau representation} is a composition representation $\tau\circ\rho_3$, where $\tau:GL_2(\mathbb{Z}[t,t^{-1}])\rightarrow GL_2(\mathbb{R})$ is an evaluation map determined by $t\mapsto t_0$ for some fixed $t_0\in\mathbb{R}$.

\end{defn}

\begin{thm}\label{fifff}
For $t_0\in \R$, the image of the specialization of the Burau representation at $t_0$ is isomorphic to the image when specializing to $t_0^{-1}$.  In particular, specializing to $t_0$ is faithful if and only if specializing to $t_0^{-1}$ is faithful.
\end{thm}

\begin{proof}
Let $\psi$ be the contragradient representation of $\rho_n$. For $w\in B_n$, if $\rho_n(w)=M(t)$ then $\psi(w)=(M(t)^{-1})^T$ where $M\in GL_{n-1}(\mathbb{Z}[t,t^{-1}])$. From Squier, there exists a matrix $J$ so that $$M(t)^*=JM(t)^{-1}J^{-1}.$$
Taking the transpose of both sides shows that $M(t^{-1})$ is conjugate to $(M(t)^{-1})^T$ by $J^T$.  Thus $\rho_n$ and $\psi$ are conjugate representations. Conjugation preserves faithfulness.

\end{proof}

\begin{prop}\label{idisc}
Specializing the Burau representation to some $t_0$ is discrete if and only if specializing to $t_0^{-1}$ is discrete.
\end{prop}

\begin{proof}
Theorem \ref{fifff} showed that $M(t_0^{-1})$ is conjugate to $(M(t_0)^{-1})^T$ by $J^T$ for every $M$ in the image of $\rho_n$. Discreteness is preserved by conjugation, inversion and transposition. So, specializing to $t_0^{-1}$ is discrete if and only if specializing to $t_0$  is discrete.
\end{proof}

\subsection{}\textbf{Subgroup Properties of $B_3$}\\

There are two well known subgroups of $B_3$ that play a vital role in the classification.
\begin{enumerate}
\item The center of $B_3$ is $Z(B_3)=\langle (\sigma_1\sigma_2)^3\rangle$ which is cyclic.

\item The normal subgroup $N= \langle a_1,a_2\rangle$ where $a_1=\sigma_1^{-1}\sigma_2$  and $a_2=\sigma_2\sigma_1^{-1}$, which is a free group on two generators. A proof of this will be shown in the proof of Theorem \ref{combind}.

\end{enumerate}

These subgroups will be used in combination with the following Lemmas and Theorem.

\begin{lem}[Long\cite{LONG}]\label{longlem} Let $\rho:B_n\rightarrow GL(V)$ be a representation and $K\triangleleft B_n$ with $K$ nontrivial and non central. If $\rho |_K$ is faithful, then $\rho$ is faithful except possibly on the center. \end{lem}

\begin{lem}\label{free2}
Every homomorphism $\phi$ on $N$ with $\phi(N)$  a free group of rank two is an isomorphism onto its image.
\end{lem}

\begin{proof}                     

Since $N$ is a free group of rank two, it is Hopfian. It is given that $\phi(N)$ is also a free group of rank two. Therefore by definition of Hopfian, $\phi$ must be an isomorphism on $N$.

\end{proof}

\begin{lem}\label{faithcen}
The Burau representation on $B_3$ is faithful on the center for all real specializations of $t$ except $t=0,\pm 1$.
\end{lem}
\begin{proof}
 The center of $B_3$ is cyclicly generated by $(\sigma_1\sigma_2)^3$, where
\[ \rho_3\left(( \sigma_1\sigma_2)^3\right)=\left( \begin{array}{cc}
t^3 & 0  \\
0 & t^3  \end{array} \right).\]

This shows that $\rho_3(Z(B_3))$ is a free group on one generator when $t\neq \pm1,0$. So $\rho_3$ is faithful on $Z(B_3)$.
\end{proof}

\begin{cor}
 Away from 0 and $\pm 1$, if a specialization the Burau representation is faithful on $N$, then it is faithful on all of $B_3$.
\end{cor}
\begin{proof}
Lemma 2.7 proves that the specialization is faithful on the center. Since $N$ is a normal subgroup of $B_3$, Lemma 2.5 guarantees that the specialization is faithful on the rest of $B_3$. \end{proof}


\begin{thm}\label{disc}
If $\rho_3$ is discrete on $N$, then $\rho_3$ is discrete on all of $B_3$. 
\end{thm}


\begin{proof}
Assume for a contradiction that  $\{\gamma_k\}$ is a sequence in $\rho_3(B_3)$ converging to the identity but $\gamma_k\neq Id$ for all $k$. Then for every fixed $\phi\in \rho_3(N)$, the commutator sequence $\{[\phi,\gamma_k]\}$ also converges to the identity. Since $N$ is normal  and $\rho_3(N)$ is discrete, then $\{[\phi, \gamma_k]\}\subseteq \rho_3(N)$ and for some $n_0\in\mathbb{N}$, $[\phi, \gamma_k]=Id$ for all $k>k_0$. This gives that for all $k>n_0$, $$\phi\gamma_k=\gamma_k\phi.$$ 

 
Because $B_3$ is not virtually solvable,  $\rho_3(B_3)$ is non-elementary and $\rho_3|_N$ is discrete, there exists a hyperbolic element $\eta$ of $\rho_3(N)$ so that $Fix(\eta)$ is disjoint from $Fix(\gamma_k)\cup Fix(\gamma_k^2)$ for all $k>n_0$ \cite[p. 606]{RAT}. Since $\eta$ is hyperbolic, let $\{y_1,y_2\}=Fix(\eta)$. As shown above, $\eta$ and $\gamma_k$ must commute, which gives that $$\eta\gamma_k(y_i)=\gamma_k\eta(y_i)=\gamma_k(y_i).$$

This implies that $\gamma_k(y_i)$ is a fixed point of $\eta$, which means $\gamma_k(y_i)\in\{y_1,y_2\}$. However, if $\gamma_k(y_1)=y_1$ then $y_1$ is a fixed point of both $\gamma_k$ and $\eta$. If $\gamma_k(y_1)=y_2$ then $\gamma_k(y_2)=y_1$ which forces $\gamma_k^2(y_1)=y_1$. In either case, $y_1$ is a fixed point of $\gamma_k$ or $\gamma_k^2$ which contradicts the choice of $\eta$. 
 \end{proof}
 
 \noindent

\noindent\textbf{Remark:} Theorem \ref{disc} can be generalized with effectively the same proof, but is a slight tangent from the realm of braids and requires a bit of hyperbolic geometry.\\

\noindent\textbf{Theorem \ref{disc} generalized:} \textit{Let $G$ be a group that is not virtually solvable and $K$ a non central normal subgroup of $G$. If $\rho: G\rightarrow Isom^+(\mathbb{H}^n)$ is a homomorphism so that $\rho(G)$ is non-elementary, $\rho|_K$ is discrete, and $\rho(K)\not\subset Ker(\rho)$ then $\rho$ is discrete on all of $G$.}

\section{Main Result}

\begin{proof}[Proof of Theorem \ref{combind}]
$\\$
With Theorem \ref{disc} in sight, the image of the normal subgroup $N$ under $\rho_3$ is generated by the following two matrices.
\[\rho_3(a_1)=  \left( \begin{array}{cc}
\frac{t-1}{t} & -1  \\
t & -t  \end{array} \right)\hspace{1cm} \rho_3(a_2)= \left(\begin{array}{cc}
-\frac{1}{t} & \frac{1}{t}  \\
-1 & 1-t  \end{array} \right)\]

Next, define $\iota, x$ and $y$ as follows
\[ \iota=\left( \begin{array}{cc}
1 & 0  \\
0 & -1  \end{array} \right),\]

\[  x=\iota^{-1}\rho_3(a_2)\iota= \left(\begin{array}{cc}
-\frac{1}{t} & -\frac{1}{t}  \\
1 & 1-t  \end{array} \right)\text{, \hspace{1mm} and   \hspace{1mm}} y= \iota^{-1}\rho_3(a_1)\iota=\left( \begin{array}{cc}
\frac{t-1}{t} & 1  \\
-t & -t  \end{array} \right).\]

 Let $S_t$ denote the specialization of $\rho_3$ for some fixed $t\in \mathbb{R}$ and $M=\langle x,y\rangle$ in $GL_2(\mathbb{R})$.  Since $S_t(N)$ is conjugate to $M$ by $\iota$, the discreteness of  $S_t(N)$ is completely determined by the discreteness of $M$.  
 
 Let $D^2=\mathbb{H}^2\cup S^1_{\infty}$ denote the Poincare disk model of the upper half plane. Notice that $x,y\in SL_2(\Z[t,\frac{1}{t}])$ and tr$(x)= $ tr$(y)= -\frac{1}{t}+1-t$. By comparing $( -\frac{1}{t}+1-t)^2$ to $4$, both $x$ and $y$ act as isometries of the following type:
\begin{enumerate}[a)]
\item Hyperbolic when $t < 0$ or $0 < t < \frac{3 - \sqrt{5}}{2}$ or $ t > \frac{3 + \sqrt{5}}{2}$,
\item Elliptic when $ \frac{3 - \sqrt{5}}{2}<t<\frac{3 + \sqrt{5}}{2}$,
\item  Parabolic when $t= \frac{3 \pm \sqrt{5}}{2}$.
 \end{enumerate}

Consider the following cases on $t\in \R$.
$\\$
\noindent
\textbf{Case 1)} Let $t<0$.
$\\$
In this range of $t$, both $x$ and $y$ act as hyperbolic isometries on  $D^2$. Consider the following images of $\infty$: 
$$y^{-1}(\infty)=-1  \text{, \hspace{1mm} and   \hspace{1mm}} xy^{-1}(\infty)=0$$
$$yxy^{-1}(\infty)=-\frac{1}{t}=x(\infty).$$



The unshaded region of Figure \ref{Maction} is a fundamental domain for the  action of $M$ on $D^2$. So $\mathbb{H}^2/M$ is a punctured torus, showing that $M$ and  $S_t(N)$ are discrete, and $M$ is a free group of rank 2.   By Theorem \ref{disc}, since $S_t$ is discrete on $N$ then it is discrete on all of $B_3$.

\begin{figure}[h!]
\centering
\begin{tikzpicture}
  \begin{scope}
    \clip (0,0) circle (2cm);
      \draw[lgray,fill=lgray] (-3,0) circle (2.1cm); 
        \draw[lgray,fill=lgray] (0,-3) circle (2.1cm); 
          \draw[lgray,fill=lgray] (3,0) circle (2.1cm); 
        \draw[lgray,fill=lgray] (0,3) circle (2.1cm); 
   
    \draw[fill] (0,0) circle [radius=0.06] node[right]{$p$};
  \end{scope}
  \draw[thick] (0,0) circle (2cm);
     \draw[orange,line width=.6mm] (-1.4,-1.38) arc (-42:45:2.03);
     \draw[orange,line width=.6mm] (1.45,-1.4) arc (225:135:2.0);
     \draw[blue,line width=.6mm] (-1.4,1.45) arc (225:315:2.0);
     \draw[blue,line width=.6mm] (-1.4,-1.45) arc (135:45:2.0);
  \fill[black] (1.4, 1.4) circle (.1cm)  node[right]{  \hspace{2mm}$xyx^{-1}(\infty)=\frac{1}{-t}=x(\infty)$};
\fill[black] (-1.4, 1.4) circle (.1cm)  node[left]{ \hspace{2mm} $\infty$\hspace{2mm}};
\fill[black] (-1.4, -1.4) circle (.1cm)  node[left]{ \hspace{2mm} $y^{-1}(\infty)=-1$};
\fill[black] (1.4, -1.4) circle (.1cm)  node[right]{ \hspace{2mm} $xy^{-1}(\infty)=0$};

\end{tikzpicture}
\caption{$D^2$ with geodesics connecting images of $\infty$.}
\label{Maction}
\end{figure}
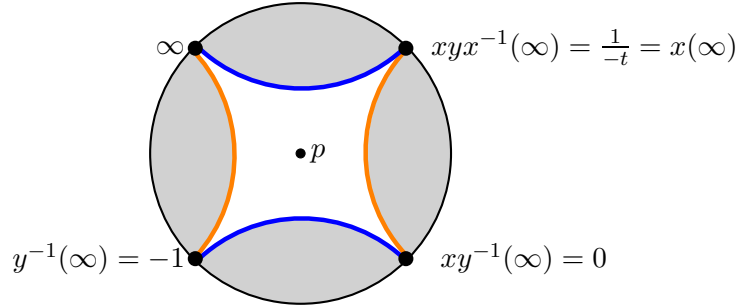


By Lemma \ref{free2}, since $M\cong S_t(N)$, $S_t$ is faithful on $N$ and $N$ is free of rank two. By Lemmas \ref{longlem} and \ref{faithcen}, $S_t$ is faithful on all of $B_3$.
$\\$

\noindent
\textbf{Case 2)} Let $t= \frac{3+\sqrt{5}}{2}$.

For this value of $t$, $x$, $y$ and $yx^{-1}$ are parabolic isometrics. Let $x^{-1}_f, y_f$ and $z_f$ denote fixed points of $x^{-1}$, $y$ and $yx^{-1}$ respectively. By computing eigenvectors, these fixed point are
$$x^{-1}_f=\frac{-1+\sqrt{5}}{2}, \hspace{5mm} y_f=\frac{1-\sqrt{5}}{2}, \hspace{5mm} z_f=\frac{-7+3\sqrt{5}}{2}.$$

Figure \ref{atu} shows a fundamental domain for the action of $M$ on $D^2$, again showing that $\mathbb{H}^2/M$ is a punctured torus. By the same arguments as in case 1,  $S_t$ is discrete and faithful on all of $B_3$.
\begin{figure}[h]
    \centering
    
    \begin{tikzpicture}
\begin{scope}
\clip(0,0) circle (2cm);
\draw[fill=lgray] (0,0) circle (2cm);
\draw[thick,orange, fill =white] (2,2) circle (2cm);
\draw[thick,blue, fill =white] (-2,2) circle (2cm);
\draw[thick,orange, fill =white] (2,-2) circle (2cm);
\draw[thick,blue, fill =white] (-2,-2) circle (2cm);
      \end{scope}
      \fill[black] (0, -2) circle (.1cm)  node[below]{ $z_f$};
       \fill[black] (0, 2) circle (.1cm)  node[above]{ $x^{-1}(z_f)$};
       \fill[black] (2, 0) circle (.1cm)  node[right]{ $y_f$};
        \fill[black] ( -2,0) circle (.1cm)  node[left]{ $x_f^{-1}$};
        
          \draw[thick,->] (-2.2,-.5) arc (-95:95:.5);
           \draw[thick,->] (2.1,.5) arc (95:275:.5);
            \draw[thick,->] (-.5,-2.1) arc (185:-5:.5);
            
            \draw[white] (1.2,1.2) circle [radius=0.01] node[left,black]{$A$};
             \draw[white] (1.2,-1.2) circle [radius=0.01] node[left,black]{$B$};
             \draw[white] (-1.2,-1.2) circle [radius=0.01] node[right,black]{$C$};
              \draw[white] (-1.2,1.2) circle [radius=0.01] node[right,black]{$D$};
            
                   \draw[thick] (0,0) circle (2cm);
\end{tikzpicture}

    \caption{The shaded region is the fundamental domain for the action of $M$ on $D^2$ when $t=\frac{3+\sqrt{5}}{2}$.}
    \label{atu}
\end{figure}


$\\$
\noindent
\textbf{Case 3)} Let $t> \frac{3+\sqrt{5}}{2}$.

In this region, both $x$, $y$ and $x^{-1}$ act as hyperbolic isometries on the $D^2$. As shown in Case 2, the fixed points of $x^{-1}$, $y$ and $yx^{-1}$ are distinct when $t=\frac{3+\sqrt{5}}{2}$. 
If there exists a $t$ so that any two of $x^{-1}$, $y$ or $x^{-1}y$ shared a fixed point then, then both $x^{-1}$ and $y$ share a fixed point. In other words, $x^{-1}$ and $y$ have a common eigenvector and are simultaneously conjugate to matrices of the form
\[  x^{-1}\sim \left(\begin{array}{cc}
a & *  \\
0 & a^{-1}  \end{array} \right)\text{ \hspace{1mm} and   \hspace{1mm}} y\sim \left( \begin{array}{cc}
b & *  \\
0 & b^{-1}  \end{array} \right)\]

for some $a,b\in \R$. This forces the commutator $[x^{-1},y]$ to have the form

\[[x^{-1},y]\sim \left( \begin{array}{cc}
1 & *  \\
0 & 1 \end{array} \right),\]
 
\noindent which gives tr$([x^{-1},y])=2$. However, by direct computation, tr$([x^{-1},y])=\frac{(1 + t^2) (1 - t^2 + t^4)}{t^3}$ which is strictly greater than $2$ for $t>\frac{3+\sqrt{5}}{2}$. So as $t$ increases, all six  fixed points of $x^{-1}$, $y$ and $x^{-1}y$ remain distinct for all $t>\frac{3+\sqrt{5}}{2}$.

Let $x_{\pm}, y_{\pm}$ and $z_{\pm}$ denote the fixed points of each $x^{-1}$, $y$ and $yx^{-1}$ respectively.  Figure \ref{greaterthanu}  shows a fundamental domain for the action of $M$ on $D^2$, and $\mathbb{H}^2/M$ is again a punctured torus. So $S_t$ is discrete and faithful on $B_3$.

\begin{figure}[h]
    \centering
    
    \begin{tikzpicture}
\begin{scope}
\clip(0,0) circle (2cm);
\draw[fill=lgray] (0,0) circle (2cm);
\draw[thick,blue, fill =white] (2,2) circle (2cm);
\draw[thick,blue, fill =white] (-2,2) circle (2cm);
\draw[thick,orange, fill =white] (2,-2) circle (2cm);
\draw[thick,orange, fill =white] (-2,-2) circle (2cm);
      \end{scope}
      \fill[black] (0, -2) circle (.1cm)  node[below]{ $x_+$};
       \fill[black] (0, 2) circle (.1cm)  node[above]{ $y_+$};
       \fill[black] (2, 0) circle (.1cm)  node[right]{ $x^{-1}(z_-)$};
        \fill[black] ( -2,0) circle (.1cm)  node[left]{ $z_-$};
        
          \draw[thick,->,gray] (-1.68,1.11) arc (6:-35:1.5);
          \fill[gray] ( -1.68,1.11) circle (.08cm)  node[left]{ $z_+$};
            \draw[thick,->,gray] (.02,-1.9) arc (140:85:1.5);
             \fill[gray] ( 1.41,-1.41) circle (.08cm)  node[below]{   $x_-$};
              \draw[thick,->,gray] (-.05,1.93) arc (-40:-100:1.3);
                \fill[gray] (-1.35, 1.5) circle (.08cm)  node[above]{ $y_-$};
            
            \draw[white] (1.2,1.2) circle [radius=0.01] ;
             \draw[white] (1.2,-1.2) circle [radius=0.01] ;
             \draw[white] (-1.2,-1.2) circle [radius=0.01] ;
              \draw[white] (-1.2,1.2) circle [radius=0.01] ;
            
                   \draw[thick] (0,0) circle (2cm);
\end{tikzpicture}

    \caption{The shaded region is the fundamental domain for the action of $M$ on $D^2$ when $t>\frac{3+\sqrt{5}}{2}$.}
    \label{greaterthanu}
\end{figure}




$\\$
\noindent
\textbf{Case 4)} Let $0<t\leq\frac{3-\sqrt{5}}{2}$.

 Immediately from case 2, case 3, and  Corollary \ref{disc}, $S_t$ is discrete and faithful on all of $B_3$.

$\\$
\noindent
\textbf{Case 5)} Let $\frac{3-\sqrt{5}}{2} <t<\frac{3+\sqrt{5}}{2}$.

In this region, $x^{-1}$, $y$ and $yx^{-1}$ are all elliptic with the same trace $1-t-\frac{1}{t}$. Elliptic isometries are diagolizable with diagonal entries complex conjugate roots of unity. So the trace is $2\cos\theta$ for some $\theta$ which is the rotation angle for the isometry. At $t=\frac{3+\sqrt{5}}{2}$, the trace of $x^{-1}$, $y$ and $yx^{-1}$ are all equal to $-2$. To account for this negative sign, the following equation must hold 
$$-2\cos\theta=1-t-\frac{1}{t}.$$
Solving for $t$ in terms of $\theta$ gives
$$t=\frac{1+2\cos\theta\pm\sqrt{(2\cos\theta+1)^2-4}}{2}.$$

Since $t$ is real valued, the discriminant must be nonnegative, forcing

$$\cos\theta\leq-\frac{3}{2}\hspace{1cm}\text{ or} \hspace{1cm} \cos\theta\geq\frac{1}{2}.$$
Thus, the only possible rotation angles for $x^{-1}$, $y$ and $yx^{-1}$ are  $0\leq \theta \leq \frac{\pi}{3}$ or $\frac{5\pi}{3}\leq \theta\leq 2\pi$.
Consider the following cases for $\theta$.
\begin{enumerate}
\item If $\theta=d\pi$ where $d$ is irrational. 

Let $x_f$ and $y_f$ be the fixed points of $x$ and $y$ respectively. Since $y$ acts as a rotation about $y_f$, the set $\{y^i(x_f)\}_{i\in\mathbb{N}}$ lies in an $S^1$ centered at $y_f$. Since $\frac{\theta}{\pi}$ is irrational, $y^i(x_f)$ is distinct for each $i$. By compactness, $\{y^i(x_f)\}_{i\in\mathbb{N}}$ has an accumulation point, giving the orbit of $x_f$ is not discrete and the action of $M$ is not discrete.

\item If $\theta=\frac{2\pi}{n}$ for some $n\in\Z$.

Then $M$ is the triangle group with presentation $\langle x,y| x^n=y^n=(xy)^n=1\rangle$. The bounds for $\theta$ force $n\geq 6$ and all such $n$ occur from specializations of $t$ satisfying $\frac{3-\sqrt{5}}{2} <t<\frac{3+\sqrt{5}}{2}$. For $n\geq 6$, $\frac{1}{n}+\frac{1}{n}+\frac{1}{n}<1$ so $M$ is a hyperbolic triangle group and is known to be discrete.

\item  If $\theta=\frac{ 2\pi k}{m}$ for $k,m\in \Z$ relatively prime.

The classification of good orbifolds gives that $D^2/M$ can not yield a cone angle of $\frac{ 2\pi k}{m}$ for $k,m\in \Z$ relatively prime. So
 the action of $M$ is not discrete.
 

\end{enumerate}
\end{proof}

\section{Corollaries and Examples}

There is interesting faithfulness interplay between the Burau representations $\rho_3$ on $B_3$ and $\rho_4$ on $B_4$. 
The underlying reason for this interplay is the block structure of $\rho_4$ shown in the definition below.
 \[ \rho_4(\sigma_1)=\left( \begin{array}{ccc}
-t & 1 & 0 \\
0 & 1  & 0 \\
0 & 0 & 1\end{array} \right)=
\left(
\begin{array}{c|c}

 \raisebox{-15pt}{{\LARGE\mbox{{$\rho_3(\sigma_1)$}}}} &0\\[-1ex]
&0\\\hline
 0\text{\hspace{6mm}}0&1
\end{array}
\right),
\]

\[ \rho_4(\sigma_2)=\left( \begin{array}{ccc}
1 & 0 &0 \\
t & -t  & 1\\
0 & 0 & 1\end{array} \right)=
\left(
\begin{array}{c|c}

 \raisebox{-15pt}{{\LARGE\mbox{{$\rho_3(\sigma_2)$}}}} &0\\[-1ex]
&1\\\hline
 0\text{\hspace{6mm}}0&1
\end{array}
\right),
\]

\[\rho_4(\sigma_3)=\left( \begin{array}{ccc}
1 & 0  & 0 \\
0 & 1 & 0 \\
0 & t & -t  \end{array} \right).\]

One way to create an unfaithful specialization of $\rho_4$ is to "extend" an unfaithful specialization of $\rho_3$. More precisely, suppose the specialization of $\rho_3$ at $\eta$ is unfaithful, and let $K$ denote the kernel in $B_3$. We can identify $K$ as a subgroup of $B_4$ under the standard inclusion. From the block structures shown above, $\rho_4(K)$ consists of upper triangular matrices with ones along the diagonals, which is a nilpotent group as a subgroup of the Heisenberg group. Thus the upper central series finitely terminates yielding a nontrivial subgroup of $K$ that maps to the identity by $\rho_4$. Therefore, the specialization of $\rho_4$ at $\eta$ is also unfaithful.

Example 4.1 shows one method to create unfaithful specializations of $\rho_3$, which consequently are also unfaithful specializations of $\rho_4$. Because of this consequential relationship, it is perhaps more interesting to find an unfaithful specialization of $\rho_4$ that is faithful when restricted to $B_3$. Example 4.3 gives a construction of such a specialization.

\begin{ex} A method to create unfaithful specializations of $\rho_3$ on $B_3$. \end{ex}

Let $w$ be a word in $B_3$ different from $\sigma_1^k$. Let $f_w$ be a polynomial factor of the 2-1 entry of $\rho_3(w)$ and $t_w$ be a root of $f_w$. Specializing to $t=t_w$ leaves $S_{t_w}(w)$ an upper triangular matrix. Since the image of $\sigma_1$ is also upper triangular, the group $\langle S_{t_w}(\sigma_1), S_{t_w}(w) \rangle$ is solvable. Therefore, specializing to $t_w$ cannot be faithful since $B_3$ does not have  solvable subgroups. 

Some examples such $w$'s and $f_w$'s are listed here.

\begin{enumerate}
\item Let $w=\sigma_2^{-2}\sigma_1\sigma_2^{-1}$ with $f_w= -1 + t - 2 t^2 + t^3$ which has one real root.
\item Let $w=\sigma_2^5\sigma_1^2\sigma_2^{-4}\sigma_1\sigma_2^3$ and 

\noindent
$f_w=1 - 3 t + 6 t^2 - 10 t^3 + 13 t^4 - 16 t^5 + 16 t^6 - 15 t^7 + 
 12 t^8 - 8 t^9 + 5 t^{10} - 3 t^{11} + t^{12}$ which has two real roots.

\end{enumerate}

Theorem \ref{combind} proved that all real unfaithful specializations of $\rho_3$ come from the interval $(\frac{3-\sqrt{5}}{2}, \frac{3+\sqrt{5}}{2})$. Thus we can conclude that all real roots of $f_w$ must lie in the interval $(\frac{3-\sqrt{5}}{2}, \frac{3+\sqrt{5}}{2})$. This proves the following corollary.

\begin{cor}
Real roots of the 2-1 entries of  Burau matrices not in $\langle\sigma_1\rangle$ must lie in the interval $(\frac{3-\sqrt{5}}{2}, \frac{3+\sqrt{5}}{2})$.
\end{cor}

\begin{ex} An unfaithful specialization of $\rho_4$ on $B_4$.

\end{ex}

For simplification, let $x=\sigma_1\sigma_3^{-1}$ and $y=\sigma_2x\sigma_2^{-1}$. Consider the following words

$$\omega_1=x^{-1}y^2x^{-1}yxyx^2y^{-2}x^{-1}y^{-3}$$

$$\omega_2=y^{-1}xy^{-2}xy^{-1}x^{-1}y^{-1}x^{-2}y^2xy^2.$$

One can check that $\rho_4(\omega_1)\neq \rho_4(\omega_2)$. However, for $S_{t_0}$ the specialization of $\rho_4$ to $t_0=\frac{3+ \sqrt{5}}{2}$, the equality $S_{t_0}(\omega_1)=S_{t_0}(\omega_2)$ occurs. Theorem \ref{combind} proved that specializing $\rho_3$ at $t_0$ is faithful. Thus, the infidelity of $\rho_4$ at $t_0$ is truly a property of $B_4$, not a consequence of containing $B_3$.\\

Keeping inline with the previous discussions of discreteness, Squier's form easily gives the next result.

\begin{prop} The image of the specialization of the Burau representation at a quadratic algebraic integer is discrete. \end{prop}

\begin{proof}

Let $\alpha$ be a quadratic algebraic integer and $\sigma$ be the generator of the Galois group of $\mathbb{Q}(\alpha)$. The map $\sigma$ is determined by $\sigma(\alpha)=\alpha^{-1}$.  Fix arbitrary $n$ and consider the Burau representation on $B_n$ specialized at $\alpha$, and $J$  the associated Squier's form.
Let $\{A_k\}$ be a sequence of matrices in the image of this specialization and assume that $\{A_k\}$ converges to the $Id$. Each $A_k$ has entries in $\mathbb{Q}(\alpha)$, so the defining relation of Squier's form $A_k^*JA_k=J$ becomes $(A_k^{\sigma})^T=JA_k^{-1}J^{-1}$. So if $A_k\to Id$ then so does $A_k^{\sigma}$. Since $\sigma$ is \textit{the only} field automorphism, then there are only finitely many options for such entries $(A_k)_{ij}$, which means $A_k$ has to eventually be constant.
\end{proof}

\begin{cor}
The specialization of the Burau representation at $\frac{3+ \sqrt{5}}{2}$ is discrete.
\end{cor}

The number $\frac{3+ \sqrt{5}}{2}$ is particularly interesting as $\rho_3$ specialized at $\frac{3+ \sqrt{5}}{2}$ is both discrete and faithful, while specializing $\rho_4$ at $\frac{3+ \sqrt{5}}{2}$ is discrete and yet unfaithful.

\begin{center} Acknowledgments \end{center} This paper was funded in part by the NSF grant DMS 1463740. I would like to thank my graduate advisor Darren Long for significant insight and guidance on this project. I would also like to thank Steve Trettel for helpful conversations.

\bibliography{bib1}{}
\bibliographystyle{plain}

\end{document}